\documentclass[english]{article}
\usepackage[T1]{fontenc}
\usepackage[latin9]{inputenc}
\usepackage{geometry}
\usepackage{amsthm}
\usepackage{amsmath}
\usepackage{amssymb}
\usepackage{esint}

\makeatletter

\newcommand{\noun}[1]{\textsc{#1}}

\theoremstyle{plain}
\newtheorem{thm}{\protect\theoremname}
  \theoremstyle{definition}
  \newtheorem{defn}[thm]{\protect\definitionname}
  \theoremstyle{plain}
  \newtheorem{lem}[thm]{\protect\lemmaname}

\usepackage{babel}

\makeatother

\usepackage{babel}
  \providecommand{\definitionname}{Definition}
  \providecommand{\lemmaname}{Lemma}
\providecommand{\theoremname}{Theorem}

\begin{document}

\title{Weak Solution of a Doubly Degenerate Parabolic Equation}

\author{Di Kang, Tharathep Sangsawang and Jialun Zhang}
\maketitle
\begin{abstract}
This paper studies a fourth-order, nonlinear, doubly-degenerate parabolic equation derived from the thin film equation in spherical geometry. A regularization method is used to study the equation and several useful estimates are obtained. The main result of this paper is a proof of the existence of a weak solution of the equation in a weighted Sobolev space. 
\end{abstract}

\section{Introduction}

The classical thin film equation of the form 
\begin{equation}
u_{t}+\left(|u|^{n}u_{xxx}\right)_{x}=0,\label{eq:Bernis eqn}
\end{equation}
describes the behavior of a thin viscous film on a flat surface under
the effect of surface tension. In 1990, Bernis and Friedman \cite{bernis1990higher}
published their pioneering analytical result on equation (\ref{eq:Bernis eqn}).
This equation is a degenerate parabolic equation since $|u|^{n}=0$
when $u=0$. Therefore, Bernis and Friedman \cite{bernis1990higher} used a regularization
method and modified the equation to 
\[
u_{t}+\left(\left(|u|^{n}+\epsilon\right)u_{xxx}\right)_{x}=0,
\]
with a small positive parameter $\epsilon$, on which the Schauder estimates in \cite{friedman1958interior} can be
applied. Bernis and Friedman \cite{bernis1990higher} proved the existence
of non-negative weak solutions of equation (\ref{eq:Bernis eqn}) when
$n\geq1$. In addition, they showed that for $n\geq4$, with a positive
initial condition, there exists a unique positive classical solution.
In 1994, Bertozzi et al. \cite{bertozzi1994singularities} generalized this result to $n\geq\frac{7}{2}$. In 1995, Beretta et al. \cite{beretta1995nonnegative}
showed that this positivity-preserving property holds for almost every
time $t$ in the case $n\geq2$. Regarding the long time behavior, Carrillo
and Toscani \cite{carrillo2002long} proved the convergence to a self-similar
solution for equation (\ref{eq:Bernis eqn}) with $n=1$ and Carlen
and Ulusoy \cite{carlen2007asymptotic} gave an upper bound on the
distance from the self-similar solution. A similar result on a cylindrical
surface was obtained in \cite{burchard2012convergence}.

In this paper, we consider the related equation 
\begin{equation}
u_{t}+\left(|u|^{n}\left(1-x^{2}\right)\left(\left(1-x^{2}\right)u_{x}\right)_{xx}\right)_{x}=0,\quad \text{in}\ \Omega\times(0,T_{0})\label{eq:main eqn}
\end{equation}
where $n\geq1$ and $\Omega=(-1,1)$. This equation arises from the
thin film equation in spherical geometry. In \cite{kangcoatingsphere},
the authors examined the dynamics of a thin liquid film coating
the outer surface of a sphere rotating around its vertical axis in
the presence of gravity. The evolution equation describing the thickness
of the thin film was shown to be given by 
\[
h_{t}+\frac{1}{\sin\theta}\left[\sin\theta \, h^{3}\left(a\sin\theta+b\sin\theta\cos\theta+c\left(2h+\frac{1}{\sin\theta}\left(\sin\theta \, h_{\theta}\right)_{\theta}\right)_{\theta}\right)\right]_{\theta}=0,
\]
where $\theta$ is the polar coordinate ($0<\theta<\pi$), $h(\theta,t)$ is the thickness
of the thin film and $a$, $b$ and $c$ are parameters describing
the effects of gravity, rotation and surface tension, respectively. Using the change
of variable $x=-\cos\theta$, the equation becomes:
\begin{equation}
u_{t}+\left[u^{3}\left(1-x^{2}\right)\left(a-bx+c\left(2h+\left((1-x^{2})u_{x}\right)_{x}\right)_{x}\right)\right]_{x}=0\label{eq:physical eqn} \,.
\end{equation}
Equation (\ref{eq:main eqn}) results from collecting the highest
order terms of equation (\ref{eq:physical eqn}) and replacing $u^3$ with $|u|^n$ to be able study the role of exponent $n$ as is standard practice for the thin-film equation.

Equation (\ref{eq:main eqn}) is somewhat similar to equation (\ref{eq:Bernis eqn}).
However, (\ref{eq:main eqn}) loses its parabolicity at $x=\pm1$
or when $u=0$, which makes it a doubly-degenerate parabolic
equation. In this paper, we will use a similar regularization method
to prove the existence of non-negative weak solutions of the equation.
Since the equation is doubly degenerate, the existence can be proven
only in a weighted Sobolev space.

\section{Weighted Sobolev Space}

In this section, we will give the definition of weighted Sobolev spaces
and provide two related lemmas which will be used in section three. 
\begin{defn}
\noun{{[}Weighted Sobolev space{]}} Let $\Omega$ be a bounded interval
and $w(x)$, $v_{0}(x)$ and $v_{1}(x)$ be nonnegative continuous
functions defined on $\Omega$. The weighted $L^{2}$ space with weight
$w$ is defined as 
\[
L_{w}^{2}(\Omega)=\left\{ u~\left|~\int_{\Omega}u(x)^{2}w(x)~dx<\infty\right.\right\} 
\]
where the weighted $L^{2}$ norm is given by 
\[
||u||_{L_{w}^{2}}=\left(\int_{\Omega}u(x)^{2}w(x)~dx\right)^{1/2}.
\]
The weighted Sobolev space $W^{1,p}(\Omega;v_{0},v_{1})$ is the set
of functions $u$ such that the norm 
\[
||u||_{1,p,v_{0},v_{1}}:=\left(\int_{\Omega}|u|^{p}v_{0}~dx+\int_{\Omega}|\nabla u|^{p}v_{1}~dx\right)^{\frac{1}{p}}
\]
is finite. Specifically, the weighted $H^{1}$ space with weight $w$
is defined as 
\[
H_{w}^{1}(\Omega)=\left\{ u\left|||u||_{H_{w}^{1}}=||u||_{L_{w}^{2}}+||u_{x}||_{L_{w}^{2}}<\infty\right.\right\} .
\]

\end{defn}
For the purposes of this paper, we take $\Omega=(-1,1)$ and $w(x)=1-x^{2}$.
In order to prove the existence of a weak solution in $H_{w}^{1}(\Omega)$,
we will use two lemmas. The first is the generalized Nirenberg inequality 
\cite{kohn1984first} given below. 
\begin{lem}
\label{lem:nirenberg}Let $p$, $q$, $r$, $\alpha$, $\beta$, $\gamma$,
$\sigma$ and $a$ be real numbers satisfying $p,q\geq1$, $r>0$
, $0\leq a\leq1$, $\gamma=a\sigma+(1-a)\beta$, $\frac{1}{p}+\frac{\alpha}{n}>0$,
$\frac{1}{q}+\frac{\beta}{n}>0$ and $\frac{1}{r}+\frac{\gamma}{n}>0$.
There exists a positive constant $C$ such that the following inequality
holds for all $u\in C_{0}^{\infty}(\mathbb{\mathbb{R}}^{n})$ ($n\geq1$)
\[
\left\Vert |x|^{\gamma}u\right\Vert _{L^{r}}\leq C\left\Vert |x|^{\alpha}|Du|\right\Vert _{L^{p}}^{a}\left\Vert |x|^{\beta}u\right\Vert _{L^{q}}^{1-a}
\]
if and only if 
\[
\frac{1}{r}+\frac{\gamma}{n}=a\left(\frac{1}{p}+\frac{\alpha-1}{n}\right)+(1-a)\left(\frac{1}{q}+\frac{\beta}{n}\right)
\]
and 
\[
\left\{ \begin{array}{ccccc}
1\leq\alpha-\sigma & \text{if} & a>0\\
\alpha-\sigma\leq1 & \text{if} & a>0 & \text{and} & \frac{1}{p}+\frac{\alpha-1}{n}=\frac{1}{r}+\frac{\gamma}{n}.
\end{array}\right.
\]

\end{lem}
The next lemma is an embedding theorem from weighted Sobolev spaces
to weighted continuous function spaces \cite{brown1992embeddings}.
Before introducing the lemma, we give several more definitions. 
\begin{defn}
\noun{{[}Weighted H\"older space{]}} Let $\Omega$ be a bounded interval
and $w_{0}(x)$ and $w_{1}(x)$ be nonnegative continuous functions
defined on $\Omega$. The weighted H\"older space $C^{\lambda}(\Omega;w_{0},w_{1})$
is the set of functions $u$ such that the norm 
\[
||u||_{C(\Omega;w_{0,},w_{1})}:=\sup_{t\in\Omega}|u(t)w_{0}(t)|+\sup_{t\in\Omega}H_{\lambda}(t,u)
\]
is finite, where 
\[
H_{\lambda}(t,u)=\sup_{s\neq t}\left(w_{1}(s)\frac{|u(s)-u(t)|}{|s-t|^{\lambda}}\right).
\]
With the definitions above, we have the following lemma.\end{defn}
\begin{lem}
\label{lem:embedding}\noun{{[}Embedding Theorem{]}} Let $p\in(n,\infty)$,
$\lambda\in(0,1-\frac{n}{p}]$, and let $w_{i}(x),$ $v_{i}(x)$ and
$r(t)$ be nonnegative continuous functions for $i=1,2$. Define 
\[
D_{w}(t):=\sup_{s\in B(t,r(t))}\frac{w(s)}{w(t)},\quad C_{w}(t):=\inf_{s\in B(t,r(t))}\frac{w(s)}{w(t)},\quad S_{j}^{\lambda i}\{r,w_{i},v_{j}\}(t):=r(t)^{pj-n-\lambda pi}\frac{w_{i}^{p}(t)}{v_{j}(t)}.
\]
If 
\[
\sup_{t}\left(\frac{D_{w_{0}}^{p}(t)}{C_{v_{i}}(t)}S_{i}\{r,w_{0},v_{i}\}(t)\right)<\infty,\ i=0,1
\]
and 
\[
\sup_{t}\left(\frac{D_{w_{1}}^{p}(t)}{C_{v_{i}}(t)}S_{i}\{r,w_{1},v_{i}\}(t)\right)<\infty,\ i=0,1
\]
then 
\[
W^{1,p}(\Omega;v_{0},v_{1})\hookrightarrow C^{\lambda}(\Omega;w_{0},w_{1}).
\]

\end{lem}

\section{Regularization and Existence of Weak Solution}

In this section, we study the equation 
\[
u_{t}+\left(|u|^{n}(1-x^{2})\left((1-x^{2})u_{x}\right)_{xx}\right)_{x}=0,\ x\in\Omega=(-1,1),\ t\in(0,T_{0})
\]
where $n\geq1$ and $T_{0}>0$, with the initial condition 
\begin{equation}
u(x,0)=u_{0}(x)\in H^{1}(-1,1)\label{eq:IC}
\end{equation}
and boundary conditions 
\begin{equation}
(1-x^{2})u_{x}=(1-x^{2})\left((1-x^{2})u_{x}\right)_{xx}=0\ \text{at}\ x=\pm1.\label{eq:BC}
\end{equation}
Since this equation is doubly degenerate, we use $\epsilon$ and $\delta$
to regularize the equation as 
\begin{equation}
(u_{\epsilon,\delta})_{t}+\left(\left((u_{\epsilon,\delta})^{n}+\epsilon\right)(1-x^{2}+\delta)\left((1-x^{2}+\delta)(u_{\epsilon,\delta})_{x}\right)_{xx}\right)_{x}=0,\ x\in\Omega=(-1,1),\ t\in(0,T_{0})\label{eq:regularized eqn}
\end{equation}
where $\epsilon,\delta>0$. The corresponding regularized initial
and boundary conditions become 
\[
u_{\epsilon,\delta}(x,0)=u_{\epsilon,\delta0}(x)\in H^{1}(-1,1)
\]
and 
\[
(u_{\epsilon,\delta})_{x}=\left((1-x^{2}+\delta)(u_{\epsilon,\delta})_{x}\right)_{xx}=0\ \text{at}\ x=\pm1.
\]

The existence of a solution of (\ref{eq:regularized eqn}) in a small
time interval is guaranteed by the Schauder estimates in \cite{friedman1958interior}.
Now suppose that $u_{\epsilon,\delta}$ is a solution of equation
(\ref{eq:regularized eqn}) and that it is continuously differentiable
with respect to the time variable and fourth order continuously differentiable
with respect to the spatial variable. In order to get an \textit{a
priori} estimation of $u_{\epsilon,\delta}$, we multiply both sides
of equation (\ref{eq:regularized eqn}) by $\left((1-x^{2}+\delta)(u_{\epsilon,\delta})_{x}\right)_{x}$
and integrate over $\Omega\times(0,T)$. This gives us

\begin{align*}
 & \int_{0}^{T}\int_{-1}^{1}\left[\left((u_{\epsilon,\delta})^{n}+\epsilon\right)(1-x^{2}+\delta)\left((1-x^{2}+\delta)(u_{\epsilon,\delta})_{x}\right)_{xx}\right]\left((1-x^{2}+\delta)(u_{\epsilon,\delta})_{x}\right)_{x}~dxdt\\
 & +\int_{0}^{T}\int_{-1}^{1}(u_{\epsilon,\delta})_{t}\left((1-x^{2}+\delta)(u_{\epsilon,\delta})_{x}\right)_{x}~dxdt=0.
\end{align*}
Integrating by parts, we get 
\begin{align}
 & \frac{1}{2}\int_{0}^{T}\int_{-1}^{1}(1-x^{2}+\delta)\frac{d}{dt}|(u_{\epsilon,\delta})_{x}|^{2}~dxdt\nonumber \\
 & +\int_{0}^{T}\int_{-1}^{1}\left((u_{\epsilon,\delta})^{n}+\epsilon\right)(1-x^{2}+\delta)\left((1-x^{2}+\delta)(u_{\epsilon,\delta})_{x}\right)_{xx}^{2}~dxdt=0.\label{eq:estimate IBP}
\end{align}
The second term of equation (\ref{eq:estimate IBP}) is non-negative
and the first term is non-positive. It follows that 
\begin{equation}
\int_{-1}^{1}(1-x^{2}+\delta)(u_{\epsilon,\delta})_{x}^{2}~dx\le\int_{-1}^{1}(1-x^{2}+\delta)(u_{\epsilon,\delta0})_{x}^{2}~dx\leq C_{0}\label{eq:estimate ux bound}
\end{equation}
where $C_{0}>0$ is some constant. This shows that the family of functions
$\{u_{\epsilon,\delta}\}$ is uniformly bounded in the $L_{w}^{2}$
norm.

To obtain a uniform bound for $\{u_{\epsilon,\delta}\}$, we apply
lemma \ref{lem:nirenberg} and set the parameters as follows: 
\[
n=1,\ \gamma=\frac{1}{2},\ r=p=2,\ a=1,\ \sigma=\frac{1}{2},\ \alpha=\frac{3}{2}.
\]
Then we obtain the inequality 
\begin{equation}
\big\| x^{\frac{1}{2}}u\big\|_{L^{2}}\leq C\big\| x^{\frac{3}{2}}u_{x}\big\|_{L^{2}}.\label{eq:nirenberg with parameters}
\end{equation}
Combining (\ref{eq:estimate ux bound}) and (\ref{eq:nirenberg with parameters}),
we have 
\begin{align}
\int_{-1}^{1}(1-x^{2})(u_{\epsilon,\delta})^{2}dx & \leq\frac{1}{2}\int_{-1}^{0}(1+x)(u_{\epsilon,\delta})^{2}+\frac{1}{2}\int_{0}^{1}(1-x)(u_{\epsilon,\delta})^{2}dx\nonumber \\
 & \leq\frac{1}{2}C^{2}\left(\int_{-1}^{0}(1+x)^{3}(u_{\epsilon,\delta})_{x}^{2}dx+\int_{0}^{1}(1-x)^{3}(u_{\epsilon,\delta})_{x}^{2}dx\right)\nonumber \\
 & \leq C_{1}\int_{-1}^{1}(1-x^{2})(u_{\epsilon,\delta})_{x}^{2}dx\leq C_{2}.\label{eq:uL2w bound}
\end{align}
where $C,C_{1},C_{2}>0$ are constants. Hence $\{u_{\epsilon,\delta}\}$
is uniformly bounded.

Next we take $n=1$, $p=2$, $\lambda=\frac{1}{2}$, $r=1$, $w_{0}=v_{0}=1$
and $w_{1}=v_{1}=w=1-x^{2}$ in Lemma \ref{lem:embedding}. It is
easy to check that with these parameters and weights, the conditions
of Lemma \ref{lem:embedding} hold. Thus we have 
\begin{equation}
H_{w}^{1}(\Omega)\hookrightarrow C_{w}^{1/2}(\Omega),\label{eq:embedding}
\end{equation}
which means that 
\begin{equation}
\left|(1-x^{2})\left(u_{\epsilon,\delta}(x_{1},t)-u_{\epsilon,\delta}(x_{2},t)\right)\right|\leq C_{3}|x_{1}-x_{2}|^{1/2},~\forall x_{1},x_{2}\in\Omega.\label{eq:equivcont x}
\end{equation}
Using the same method as \cite[Lemma 2.1]{bernis1990higher}, we can
prove similarly that 
\begin{equation}
\left|(1-x^{2})\left(u_{\epsilon,\delta}(x,t_{1})-u_{\epsilon,\delta}(x,t_{2})\right)\right|\leq C_{3}|t_{1}-t_{2}|^{1/8},\ \forall t_{1},t_{2}\in(0,T).\label{eq:equivcont t}
\end{equation}
The inequalities (\ref{eq:equivcont x}) and (\ref{eq:equivcont t})
show the existence of an weighted upper bound on the $C_{x,t,w}^{1/2,1/8}$-norm
of $u_{\epsilon,\delta}$ that is independent of $\epsilon$ and $\delta$.
By the Arzel\`{a}-Ascoli theorem, this equicontinuous property, together with the uniformly boundedness
of the weighted $H^{1}$ norm given by (\ref{eq:estimate ux bound}) and
(\ref{eq:uL2w bound}), shows that as $\epsilon\rightarrow0$ and
$\delta\rightarrow0$, every sequence $\{u_{\epsilon,\delta}\}$ has
a subsequence $\{\tilde{u}_{\epsilon,\delta}\}$ such that 
\begin{equation}
\tilde{u}_{\epsilon,\delta}\rightarrow u\label{eq:uniformly convergence}
\end{equation}
uniformly in $\Omega\times(0,T)$. We establish in the following theorem
that this $u$ is a weak solution of (\ref{eq:main eqn}). 
\begin{thm}
Any function $u$ obtained by (\ref{eq:uniformly convergence}) has
the following properties:

(1) $u$ is continuous in weighted H\"older space in $x$ with order $\frac{1}{2}$ and
in $t$ with order $\frac{1}{8}$.

(2) $u$ satisfies the boundary conditions (\ref{eq:BC}) and initial
condition (\ref{eq:IC}),

(3) $u$ is a weak solution of equation (\ref{eq:main eqn}) in the
following sense:

\[
\int_{0}^{T}\int_{-1}^{1}u\phi_{t}~dxdt+\int\int_{P}\left(|u|^{n}\left(1-x^{2}\right)\left(\left(1-x^{2}\right)u_{x}\right)_{xx}\right)\phi_{x}~dxdt=0
\]

for all $\phi\in Lip(\Omega\times(0,T))$, $\phi=0$ at $t=0$ and
$t=T$. Here $P=\Omega\times(0,T)\backslash\{(x,t)|u=0\}$.

(4) If the initial value $u_{0}$ is non-negative, then 
\[
u\geq0.
\]
\end{thm}
\begin{proof}
Parts $(1)$ and $(2)$ of the theorem can be derived directly from (\ref{eq:uniformly convergence}).
To prove $(3)$, notice that

\[
\int_{0}^{T}\int_{-1}^{1}u_{\epsilon,\delta}\phi_{t}+\left(\left((u_{\epsilon,\delta})^{n}+\epsilon\right)(1-x^{2}+\delta)\left((1-x^{2})(u_{\epsilon,\delta})_{x}\right)_{xx}\right)\phi_{x}~dxdt=0,
\]
so we have

\[
\int_{0}^{T}\int_{-1}^{1}u_{\epsilon,\delta}\phi_{t}~dxdt+\int_{0}^{T}\int_{-1}^{1}\left((u_{\epsilon,\delta})^{n}(1-x^{2})\left((1-x^{2})(u_{\epsilon,\delta})_{x}\right)_{xx}\right)\phi_{x}~dxdt
\]
\begin{equation}
+\epsilon\int_{0}^{T}\int_{-1}^{1}(1-x^{2}+\delta)\left((1-x^{2})(u_{\epsilon,\delta})_{x}\right)_{xx}\phi_{x}~dxdt+\delta\int_{0}^{T}\int_{-1}^{1}(u_{\epsilon,\delta})^{n}\left((1-x^{2})(u_{\epsilon,\delta})_{x}\right)_{xx}\phi_{x}~dxdt=0.\label{eq:theorem proof 1}
\end{equation}
Using equation (\ref{eq:estimate IBP}), we can show that the last
two terms go to $0$ as $\epsilon\rightarrow0$ and $\delta\rightarrow0$.
We also have 
\begin{equation}
\lim_{\xi\rightarrow0}\int_{\{u<\xi\}}(u_{\epsilon,\delta})^{n}(1-x^{2})\left((1-x^{2})(u_{\epsilon,\delta})_{x}\right)_{xx}~dxdt=0.\label{eq:proof 2}
\end{equation}
From (\ref{eq:theorem proof 1}) and (\ref{eq:proof 2}), we can get
(3) of the theorem.

In order to prove (4) of the theorem, we first define function $g_{\epsilon}(s)$
and $G_{\epsilon}(s)$ as following: 
\[
g_{\epsilon}(s)=-\int_{s}^{A}\frac{dr}{|r|^{n}+\epsilon},
\]
\[
G_{\epsilon}(s)=-\int_{s}^{A}g_{\epsilon}(r)dr,
\]
where $A$ is an uniform upper bound for $u_{\epsilon,\delta}$ for
all $\epsilon$ and $\delta$. Now we multiply equation (\ref{eq:regularized eqn})
by $g_{\epsilon}(u_{\epsilon,\delta})$ and integrate over $\Omega\times(0,T)$
to get 
\begin{equation}
\int_{0}^{T}\int_{\Omega}g_{\epsilon}(u_{\epsilon,\delta})\left((u_{\epsilon,\delta})_{t}+\left(\left((u_{\epsilon,\delta})^{n}+\epsilon\right)(1-x^{2}+\delta)\left((1-x^{2}+\delta)(u_{\epsilon,\delta})_{x}\right)_{xx}\right)_{x}\right)dxdt=0.\label{eq:positivity proof integrate}
\end{equation}
Note that 
\[
G_{\epsilon}^{''}(s)=g_{\epsilon}^{'}(s)=\frac{1}{|s|^{n}+\epsilon},
\]
so we have 
\[
g_{\epsilon}(u_{\epsilon,\delta})(u_{\epsilon,\delta})_{t}=\left(G_{\epsilon}\left(u_{\epsilon,\delta}(x,t)\right)\right)_{t}.
\]
After using this and integration by parts, equation (\ref{eq:positivity proof integrate})
becomes 
\[
\int_{\Omega}G_{\epsilon}\left(u_{\epsilon,\delta}(x,T)\right)dx-\int_{\Omega}G_{\epsilon}\left(u_{\epsilon,\delta}(x,0)\right)dx+\int_{0}^{T}\int_{\Omega}\left(\left((1-x^{2}+\delta)(u_{\epsilon,\delta})_{x}\right)_{x}\right)^{2}dxdt=0.
\]
As $u_{\epsilon,\delta}(x,0)$ is bounded, we have 
\begin{equation}
\int_{\Omega}G_{\epsilon}\left(u_{\epsilon,\delta}(x,T)\right)dx<C.\label{eq:positivity proof upper bound}
\end{equation}
Assume there exists a point $(x_{0},t_{0})$ such that $u(x_{0},t_{0})<0$.
Because $u_{\epsilon,\delta}(x,t)$ uniformly converges to $u(x,t)$,
we can choose $\epsilon_{0}>0$ and $\xi>0$ such that 
\[
u_{\epsilon,\delta}(x,t_{0})<-\xi,\ \text{if}\ |x-x_{0}|<\xi,\ \epsilon<\epsilon_{0}.
\]
Then 
\begin{equation}
\lim_{\epsilon\rightarrow0}G_{\epsilon}(u_{\epsilon,\delta}(x,t_{0}))=-\lim_{\epsilon\rightarrow0}\int_{u_{\epsilon,\delta}(x,t_{0})}^{A}g_{\epsilon}(s)~ds\geq-\lim_{\epsilon\rightarrow0}\int_{-\delta}^{0}g_{\epsilon}(s)~ds=\infty.\label{eq:positivity proof divergence}
\end{equation}
The last step is because $\lim_{\epsilon\rightarrow0}g_{\epsilon}(s)=-\infty$
when $s<0$. Since (\ref{eq:positivity proof divergence}) conflicts with
(\ref{eq:positivity proof upper bound}), we have proved (4)
of the theorem. 
\end{proof}

\end{document}